\documentclass[twoside,onecolumn,centertags,final,a4paper,12pt,reqno]{amsart}
\usepackage{amsmath}
\usepackage{amsfonts}
\usepackage[left=1.5cm, bottom=3cm]{geometry}
\usepackage{cite}

\newtheorem{thm}{Theorem}
\newtheorem{lemma}{Lemma}
\newtheorem{cor}{Corollary}
\newtheorem{remark}{Remark}
\theoremstyle{definition}
\newtheorem{dfn}{Definition}
\numberwithin{equation}{section}
\allowdisplaybreaks

\begin{document}
\title[Notes on close-to-convex functions related to strip domain]{Coefficient bounds for
close-to-convex functions associated with vertical strip domain}
\author{Serap BULUT}
\address{Kocaeli University, Faculty of Aviation and Space
Sciences,Arslanbey Campus, 41285 Kartepe-Kocaeli, TURKEY}
\email{serap.bulut@kocaeli.edu.tr}

\begin{abstract}
By considering a certain univalent function in the open unit disk $\mathbb{U}
$, that maps $\mathbb{U}$ onto a strip domain, we introduce a new class of
analytic and close-to-convex functions by means of a
certain non-homogeneous Cauchy-Euler-type differential equation. We determine the coefficient bounds
for functions in this new class. Relevant
connections of some of the results obtained with those in earlier works are
also provided.
\end{abstract}

\maketitle

\section{Introduction}

Let $\mathcal{A}$ denote the class of functions of the form%
\begin{equation}
f(z)=z+\sum_{n=2}^{\infty }a_{n}z^{n}  \label{1.1}
\end{equation}%
which are analytic in the open unit disk $\mathbb{U}=\left\{ z:z\in \mathbb{C%
}\;\text{and}\;\left\vert z\right\vert <1\right\} $. We also denote by $%
\mathcal{S}$ the class of all functions in the normalized analytic function
class $\mathcal{A}$ which are univalent in $\mathbb{U}$.

For two functions $f$ and $g$, analytic in $\mathbb{U}$, we say that the
function $f$ is subordinate to $g$ in $\mathbb{U}$, and write%
\begin{equation*}
f\left( z\right) \prec g\left( z\right) \qquad \left( z\in \mathbb{U}\right)
,
\end{equation*}%
if there exists a Schwarz function $\omega $, analytic in $\mathbb{U}$, with%
\begin{equation*}
\omega \left( 0\right) =0\qquad \text{and\qquad }\left\vert \omega \left(
z\right) \right\vert <1\text{\qquad }\left( z\in \mathbb{U}\right)
\end{equation*}%
such that%
\begin{equation*}
f\left( z\right) =g\left( \omega \left( z\right) \right) \text{\qquad }%
\left( z\in \mathbb{U}\right) .
\end{equation*}%
Indeed, it is known that%
\begin{equation*}
f\left( z\right) \prec g\left( z\right) \quad \left( z\in \mathbb{U}\right)
\Rightarrow f\left( 0\right) =g\left( 0\right) \text{\quad and\quad }f\left(
\mathbb{U}\right) \subset g\left( \mathbb{U}\right) .
\end{equation*}%
Furthermore, if the function $g$ is univalent in $\mathbb{U}$, then we have
the following equivalence%
\begin{equation*}
f\left( z\right) \prec g\left( z\right) \quad \left( z\in \mathbb{U}\right)
\Leftrightarrow f\left( 0\right) =g\left( 0\right) \text{\quad and\quad }%
f\left( \mathbb{U}\right) \subset g\left( \mathbb{U}\right) .
\end{equation*}

A function $f\in \mathcal{A}$ is said to be starlike of order $\alpha $%
\thinspace $\left( 0\leq \alpha <1\right) $, if it satisfies the inequality%
\begin{equation*}
\Re \left( \frac{zf^{\prime }(z)}{f(z)}\right) >\alpha \qquad \left( z\in
\mathbb{U}\right) .
\end{equation*}%
We denote the class which consists of all functions $f\in \mathcal{A}$ that
are starlike of order $\alpha $ by $\mathcal{S}^{\ast }(\alpha )$. It is
well-known that $\mathcal{S}^{\ast }(\alpha )\subset \mathcal{S}^{\ast }(0)=%
\mathcal{S}^{\ast }\subset \mathcal{S}.$

Let $0\leq \alpha ,\delta <1.$ A function $f\in \mathcal{A}$ is said to be
close-to-convex of order $\alpha $ and type $\delta $ if there exists a
function $g\in \mathcal{S}^{\ast }\left( \delta \right) $ such that the
inequality%
\begin{equation*}
\Re \left( \frac{zf^{\prime }(z)}{g(z)}\right) >\alpha \qquad \left( z\in
\mathbb{U}\right)
\end{equation*}%
holds. We denote the class which consists of all functions $f\in \mathcal{A}$
that are close-to-convex of order $\alpha $ and type $\delta $ by $\mathcal{C%
}(\alpha ,\delta )$. This class is introduced by Libera \cite{L}.

In particular, when $\delta =0$ we have $\mathcal{C}(\alpha ,0)=\mathcal{C}%
(\alpha )$ of close-to-convex functions of order $\alpha $, and also we get $%
\mathcal{C}(0,0)=\mathcal{C}$ of close-to-convex functions introduced by
Kaplan \cite{K}. It is well-known that $\mathcal{S}^{\ast }\subset \mathcal{C%
}\subset \mathcal{S}$.

Furthermore a function $f\in \mathcal{A}$ is said to be in the class $%
\mathcal{M}\left( \beta \right) $\thinspace $\left( \beta >1\right) $ if it
satisfies the inequality%
\begin{equation*}
\Re \left( \frac{zf^{\prime }(z)}{f(z)}\right) <\beta \qquad \left( z\in
\mathbb{U}\right) .
\end{equation*}%
This class introduced by Uralegaddi et al. \cite{UGS}.

Motivated by the classes $\mathcal{S}^{\ast }(\alpha )$ and $\mathcal{M}%
\left( \beta \right) $, Kuroki and Owa \cite{KO} introduced the subclass $%
\mathcal{S}\left( \alpha ,\beta \right) $ of analytic functions $f\in
\mathcal{A}$ which is given by Definition $\ref{dfn1}$ below.

\begin{dfn}
\label{dfn1}(see \cite{KO}) Let $\mathcal{S}\left( \alpha ,\beta \right) $
be a class of functions $f\in \mathcal{A}$ which satisfy the inequality%
\begin{equation*}
\alpha <\Re \left( \frac{zf^{\prime }\left( z\right) }{f\left( z\right) }%
\right) \mathbb{<\beta }\qquad \left( z\in \mathbb{U}\right)
\end{equation*}%
for some real number $\alpha \;\left( \alpha <1\right) $ and some real
number $\beta \;\left( \beta >1\right) .$
\end{dfn}

The class $\mathcal{S}\left( \alpha ,\beta \right) $\ is non-empty. For
example, the function $f\in \mathcal{A}$ given by%
\begin{equation*}
f(z)=z\exp \left\{ \frac{\beta -\alpha }{\pi }i\int_{0}^{z}\frac{1}{t}\log
\left( \frac{1-e^{2\pi i\frac{1-\alpha }{\beta -\alpha }}t}{1-t}\right)
dt\right\}
\end{equation*}%
is in the class $\mathcal{S}\left( \alpha ,\beta \right) $.

Also for $f\in \mathcal{S}\left( \alpha ,\beta \right) $, if $\alpha \geq 0$
then $f\in \mathcal{S}^{\ast }(\alpha )$ in $\mathbb{U}$, which implies that
$f\in \mathcal{S}.$

\begin{lemma}
\label{lm1}\cite{KO} Let $f\in \mathcal{A}$ and $\alpha <1<\beta $. Then $%
f\in \mathcal{S}\left( \alpha ,\beta \right) $ if and only if%
\begin{equation*}
\frac{zf^{\prime }\left( z\right) }{f\left( z\right) }\prec 1+\frac{\beta
-\alpha }{\pi }i\log \left( \frac{1-e^{2\pi i\frac{1-\alpha }{\beta -\alpha }%
}z}{1-z}\right) \qquad \left( z\in \mathbb{U}\right) .
\end{equation*}
\end{lemma}

Lemma $\ref{lm1}$ means that the function $f_{\alpha ,\beta }:\mathbb{%
U\rightarrow C}$ defined by%
\begin{equation}
f_{\alpha ,\beta }(z)=1+\frac{\beta -\alpha }{\pi }i\log \left( \frac{%
1-e^{2\pi i\frac{1-\alpha }{\beta -\alpha }}z}{1-z}\right)  \label{1.2}
\end{equation}%
is analytic in $\mathbb{U}$ with $f_{\alpha ,\beta }(0)=1$ and maps the unit
disk $\mathbb{U}$ onto the vertical strip domain%
\begin{equation}
\Omega _{\alpha ,\beta }=\left\{ w\in \mathbb{C}:\alpha <\Re \left( w\right)
\mathbb{<\beta }\right\}  \label{1.x}
\end{equation}%
conformally.

We note that the function $f_{\alpha ,\beta }$ defined by $\left( \ref{1.2}%
\right) $ is a convex univalent function in $\mathbb{U}$ and has the form%
\begin{equation*}
f_{\alpha ,\beta }(z)=1+\sum_{n=1}^{\infty }B_{n}z^{n},
\end{equation*}%
where%
\begin{equation}
B_{n}=\frac{\beta -\alpha }{n\pi }i\left( 1-e^{2n\pi i\frac{1-\alpha }{\beta
-\alpha }}\right) \qquad \left( n=1,2,\ldots \right) .  \label{1.3}
\end{equation}

Making use of Definition $\ref{dfn1}$, Kuroki and Owa \cite{KO} proved the
following coefficient bounds for the Taylor-Maclaurin coefficients for
functions in the sublass $\mathcal{S}\left( \alpha ,\beta \right) $ of
analytic functions $f\in \mathcal{A}$.

\begin{thm}
\label{th.ko}\cite[Theorem 2.1]{KO} Let the function $f\in \mathcal{A}$ be
defined by $(\ref{1.1})$. If $f\in \mathcal{S}\left( \alpha ,\beta \right) $%
, then%
\begin{equation*}
\left\vert a_{n}\right\vert \leq \frac{\prod\limits_{k=2}^{n}\left[ k-2+%
\frac{2\left( \beta -\alpha \right) }{\pi }\sin \frac{\pi \left( 1-\alpha
\right) }{\beta -\alpha }\right] }{\left( n-1\right) !}\qquad \left(
n=2,3,\ldots \right) .
\end{equation*}
\end{thm}

Here, in our present sequel to some of the aforecited works (especially \cite%
{KO}), we first introduce the following subclasses of analytic functions.

\begin{dfn}
\label{dfn3}Let $\alpha $ and $\beta $\ be real such that $0\leq \alpha
<1<\beta $. We denote by $\mathcal{S}_{g}\left( \alpha ,\beta \right) $ the
class of functions $f\in \mathcal{A}$ satisfying%
\begin{equation*}
\alpha <\Re \left( \frac{zf^{\prime }\left( z\right) }{g\left( z\right) }%
\right) \mathbb{<\beta }\qquad \left( z\in \mathbb{U}\right) ,
\end{equation*}%
where $g\in \mathcal{S}\left( \delta ,\beta \right) $ with $0\leq \delta
<1<\beta .$
\end{dfn}

Note that for given $0\leq \alpha ,\delta <1<\beta $, $f\in \mathcal{S}%
_{g}\left( \alpha ,\beta \right) $ if and only if the following two
subordination equations are satisfied:%
\begin{equation*}
\frac{zf^{\prime }\left( z\right) }{g\left( z\right) }\prec \frac{1+\left(
1-2\alpha \right) z}{1-z}\qquad \text{and\qquad }\frac{zf^{\prime }\left(
z\right) }{g\left( z\right) }\prec \frac{1-\left( 1-2\beta \right) z}{1+z}.
\end{equation*}

\begin{remark}
$(i)$ If we let $\beta \rightarrow \infty $ in Definition $\ref{dfn3}$, then
the class $\mathcal{S}_{g}\left( \alpha ,\beta \right) $ reduces to the
class $\mathcal{C}(\alpha ,\delta )$ of close-to-convex functions of order $%
\alpha $ and type $\delta $.

\noindent $(ii)$ If we let $\delta =0,\;\beta \rightarrow \infty $ in
Definition $\ref{dfn3}$, then the class $\mathcal{S}_{g}\left( \alpha ,\beta
\right) $ reduces to the class $\mathcal{C}(\alpha )$ of close-to-convex
functions of order $\alpha $.

\noindent $(iii)$ If we let $\alpha =\delta =0,\;\beta \rightarrow \infty $
in Definition $\ref{dfn3}$, then the class $\mathcal{S}_{g}\left( \alpha
,\beta \right) $ reduces to the close-to-convex functions class $\mathcal{C}$%
.
\end{remark}

Using $\left( \ref{1.x}\right) $ and by the principle of subordination, we
can immediately obtain Lemma $\ref{lm}$.

\begin{lemma}
\label{lm}Let $\alpha ,\beta $\ and $\delta $\ be real numbers such that $%
0\leq \alpha ,\delta <1<\beta $\ and let the function $f\in \mathcal{A}$ be
defined by $(\ref{1.1})$. Then $f\in \mathcal{S}_{g}\left( \alpha ,\beta
\right) $ if and only if%
\begin{equation*}
\frac{zf^{\prime }\left( z\right) }{g\left( z\right) }\prec f_{\alpha ,\beta
}(z)
\end{equation*}%
where $f_{\alpha ,\beta }(z)$ is defined by $\left( \ref{1.2}\right) $.
\end{lemma}

\begin{dfn}
\label{dfn4}A function $f\in \mathcal{A}$ is said to be in the class $%
\mathcal{B}_{g}\left( \alpha ,\beta ;\rho \right) $ if it satisfies the
following non-homogenous Cauchy-Euler differential equation:%
\begin{equation*}
z^{2}\frac{d^{2}w}{dz^{2}}+2\left( 1+\rho \right) z\frac{dw}{dz}+\rho \left(
1+\rho \right) w=\left( 1+\rho \right) \left( 2+\rho \right) \varphi (z)
\end{equation*}%
\begin{equation*}
\left( w=f(z)\in \mathcal{A},\;\varphi \in \mathcal{S}_{g}\left( \alpha
,\beta \right) ,\;g\in \mathcal{S}\left( \delta ,\beta \right) ,\;0\leq
\alpha ,\delta <1<\beta ,\;\rho \in \mathbb{R}\backslash \left( -\infty ,-1%
\right] \right) .
\end{equation*}
\end{dfn}

\begin{remark}
$(i)$ If we let $\beta \rightarrow \infty $ in Definition $\ref{dfn4}$, then
we get the class $\mathcal{B}_{g}\left( \alpha ;\rho \right) $ which
consists of functions $f\in \mathcal{A}$ satisfying%
\begin{equation*}
z^{2}\frac{d^{2}w}{dz^{2}}+2\left( 1+\rho \right) z\frac{dw}{dz}+\rho \left(
1+\rho \right) w=\left( 1+\rho \right) \left( 2+\rho \right) \varphi (z)
\end{equation*}%
\begin{equation*}
\left( \varphi \in \mathcal{C}(\alpha ,\delta ),\;0\leq \alpha ,\delta
<1,\;\rho \in \mathbb{R}\backslash \left( -\infty ,-1\right] \right) .
\end{equation*}

\noindent $(ii)$ If we let $\delta =0,\;\beta \rightarrow \infty $ in
Definition $\ref{dfn4}$, then we get the class $\mathcal{H}_{g}\left( \alpha
;\rho \right) $ which consists of functions $f\in \mathcal{A}$ satisfying%
\begin{equation*}
z^{2}\frac{d^{2}w}{dz^{2}}+2\left( 1+\rho \right) z\frac{dw}{dz}+\rho \left(
1+\rho \right) w=\left( 1+\rho \right) \left( 2+\rho \right) \varphi (z)
\end{equation*}%
\begin{equation*}
\left( \varphi \in \mathcal{C}(\alpha ),\;0\leq \alpha <1,\;\rho \in \mathbb{%
R}\backslash \left( -\infty ,-1\right] \right) .
\end{equation*}

\noindent $(iii)$ If we let $\alpha =\delta =0,\;\beta \rightarrow \infty $
in Definition $\ref{dfn4}$, then we get the class $\mathcal{M}_{g}\left(
\rho \right) $ which consists of functions $f\in \mathcal{A}$ satisfying%
\begin{equation*}
z^{2}\frac{d^{2}w}{dz^{2}}+2\left( 1+\rho \right) z\frac{dw}{dz}+\rho \left(
1+\rho \right) w=\left( 1+\rho \right) \left( 2+\rho \right) \varphi (z)
\end{equation*}%
\begin{equation*}
\left( \varphi \in \mathcal{C},\;\rho \in \mathbb{R}\backslash \left(
-\infty ,-1\right] \right) .
\end{equation*}
\end{remark}

The coefficient problem for close-to-convex functions studied many authors
in recent years, (see, for example \cite{B, BHG, R, SAS, SXW, UNR, UNAR}).
Upon inspiration from the recent work of Kuroki and Owa \cite{KO} the aim of
this paper is to obtain coefficient bounds for the Taylor-Maclaurin
coefficients for functions in the function classes $\mathcal{S}_{g}\left(
\alpha ,\beta \right) $ and $\mathcal{B}_{g}\left( \alpha ,\beta ;\rho
\right) $ of analytic functions which we have introduced here. Also we
investigate Fekete-Szeg\"{o} problem for functions belong to the function
class $\mathcal{S}_{g}\left( \alpha ,\beta \right) $.

\section{Coefficient bounds}

In order to prove our main results (Theorems $\ref{thm1}$ and $\ref{thm2}$
below), we first recall the following lemma due to Rogosinski \cite{R2}.

\begin{lemma}
\label{lm2}Let the function $\mathfrak{g}$ given by%
\begin{equation*}
\mathfrak{g}\left( z\right) =\sum_{k=1}^{\infty }\mathfrak{b}_{k}z^{k}\qquad
\left( z\in \mathbb{U}\right)
\end{equation*}%
be convex in $\mathbb{U}.$ Also let the function $\mathfrak{f}$ given by%
\begin{equation*}
\mathfrak{f}(z)=\sum_{k=1}^{\infty }\mathfrak{a}_{k}z^{k}\qquad \left( z\in
\mathbb{U}\right)
\end{equation*}%
be holomorphic in $\mathbb{U}.$ If%
\begin{equation*}
\mathfrak{f}\left( z\right) \prec \mathfrak{g}\left( z\right) \qquad \left(
z\in \mathbb{U}\right) ,
\end{equation*}%
then%
\begin{equation*}
\left\vert \mathfrak{a}_{k}\right\vert \leq \left\vert \mathfrak{b}%
_{1}\right\vert \qquad \left( k=1,2,\ldots \right) .
\end{equation*}
\end{lemma}

We now state and prove each of our main results given by Theorems $\ref{thm1}
$ and $\ref{thm2}$ below.

\begin{thm}
\label{thm1}Let $\alpha ,\beta $\ and $\delta $\ be real numbers such that $%
0\leq \alpha ,\delta <1<\beta $\ and let the function $f\in \mathcal{A}$ be
defined by $(\ref{1.1})$. If $f\in \mathcal{S}_{g}\left( \alpha ,\beta
\right) $, then%
\begin{eqnarray*}
\left\vert a_{n}\right\vert &\leq &\frac{\prod\limits_{k=2}^{n}\left[ k-2+%
\frac{2\left( \beta -\delta \right) }{\pi }\sin \frac{\pi \left( 1-\delta
\right) }{\beta -\delta }\right] }{n!} \\
&&+\frac{2\left( \beta -\alpha \right) }{n\pi }\sin \frac{\pi \left(
1-\alpha \right) }{\beta -\alpha }\left( 1+\sum_{j=1}^{n-2}\frac{%
\prod\limits_{k=2}^{n-j}\left[ k-2+\frac{2\left( \beta -\delta \right) }{\pi
}\sin \frac{\pi \left( 1-\delta \right) }{\beta -\delta }\right] }{\left(
n-j-1\right) !}\right) \qquad \left( n=2,3,\ldots \right) ,
\end{eqnarray*}%
where $g\in \mathcal{S}\left( \delta ,\beta \right) .$
\end{thm}

\begin{proof}
Let the function $f\in \mathcal{S}_{g}\left( \alpha ,\beta \right) $ be of
the form $\left( \ref{1.1}\right) $. Therefore, there exists a function%
\begin{equation}
g(z)=z+\sum_{n=2}^{\infty }b_{n}z^{n}\in \mathcal{S}\left( \delta ,\beta
\right)  \label{2.2}
\end{equation}%
so that%
\begin{equation}
\alpha <\Re \left( \frac{zf^{\prime }\left( z\right) }{g\left( z\right) }%
\right) \mathbb{<\beta }.  \label{2.3}
\end{equation}%
Note that by Theorem $\ref{th.ko}$, we have%
\begin{equation}
\left\vert b_{n}\right\vert \leq \frac{\prod\limits_{k=2}^{n}\left[ k-2+%
\frac{2\left( \beta -\delta \right) }{\pi }\sin \frac{\pi \left( 1-\delta
\right) }{\beta -\delta }\right] }{\left( n-1\right) !}\qquad \left(
n=2,3,\ldots \right) .  \label{2.a}
\end{equation}%
Let us define the function $p(z)$ by%
\begin{equation}
p(z)=\frac{zf^{\prime }\left( z\right) }{g\left( z\right) }\qquad (z\in
\mathbb{U}).  \label{2.7}
\end{equation}%
Then according to the assertion of Lemma $\ref{lm}$, we get%
\begin{equation}
p(z)\prec f_{\alpha ,\beta }(z)\qquad (z\in \mathbb{U}),  \label{2.8}
\end{equation}%
where $f_{\alpha ,\beta }(z)$ is defined by $\left( \ref{1.2}\right) $.
Hence, using Lemma $\ref{lm2}$, we obtain%
\begin{equation}
\left\vert \frac{p^{\left( m\right) }\left( 0\right) }{m!}\right\vert
=\left\vert c_{m}\right\vert \leq \left\vert B_{1}\right\vert \qquad \left(
m=1,2,\ldots \right) ,  \label{2.9}
\end{equation}%
where%
\begin{equation}
p(z)=1+c_{1}z+c_{2}z^{2}+\cdots \qquad (z\in \mathbb{U})  \label{2.10}
\end{equation}%
and (by $\left( \ref{1.3}\right) $)%
\begin{equation}
\left\vert B_{1}\right\vert =\left\vert \frac{\beta -\alpha }{\pi }i\left(
1-e^{2\pi i\frac{1-\alpha }{\beta -\alpha }}\right) \right\vert =\frac{%
2\left( \beta -\alpha \right) }{\pi }\sin \frac{\pi \left( 1-\alpha \right)
}{\beta -\alpha }.  \label{2.c}
\end{equation}%
Also from $\left( \ref{2.7}\right) $, we find%
\begin{equation}
zf^{\prime }(z)=p(z)g(z).  \label{2.11}
\end{equation}%
Since $a_{1}=b_{1}=1$, in view of $\left( \ref{2.11}\right) $, we obtain%
\begin{equation}
na_{n}-b_{n}=c_{n-1}+c_{n-2}b_{2}+\cdots
+c_{1}b_{n-1}=c_{n-1}+\sum_{j=1}^{n-2}c_{j}b_{n-j}\qquad \left( n=2,3,\ldots
\right) .  \label{2.12}
\end{equation}%
Now we get from $\left( \ref{2.a}\right) ,\left( \ref{2.9}\right) $ and $%
\left( \ref{2.12}\right) ,$%
\begin{eqnarray*}
\left\vert a_{n}\right\vert &\leq &\frac{\prod\limits_{k=2}^{n}\left[ k-2+%
\frac{2\left( \beta -\delta \right) }{\pi }\sin \frac{\pi \left( 1-\delta
\right) }{\beta -\delta }\right] }{n!} \\
&&+\frac{\left\vert B_{1}\right\vert }{n}\left( 1+\sum_{j=1}^{n-2}\frac{%
\prod\limits_{k=2}^{n-j}\left[ k-2+\frac{2\left( \beta -\delta \right) }{\pi
}\sin \frac{\pi \left( 1-\delta \right) }{\beta -\delta }\right] }{\left(
n-j-1\right) !}\right) \quad \left( n=2,3,\ldots \right) .
\end{eqnarray*}%
This evidently completes the proof of Theorem $\ref{thm1}.$
\end{proof}

\begin{remark}
It is worthy to note that the inequality obtained for $\left\vert
a_{n}\right\vert $ in Theorem $\ref{thm1}$ is also valid when $\alpha
,\delta <1<\beta $\ by Theorem $\ref{th.ko}$.
\end{remark}

Letting $\beta \rightarrow \infty $ in Theorem $\ref{thm1}$, we have the
coefficient bounds for close-to-convex functions of order $\alpha $ and type
$\delta $.

\begin{cor}
\cite{L} Let $\alpha $\ and $\delta $\ be real numbers such that $0\leq
\alpha ,\delta <1$\ and let the function $f\in \mathcal{A}$ be defined by $(%
\ref{1.1})$. If $f\in \mathcal{C}(\alpha ,\delta )$, then%
\begin{equation*}
\left\vert a_{n}\right\vert \leq \frac{2\left( 3-2\delta \right) \left(
4-2\delta \right) \cdots \left( n-2\delta \right) }{n!}\left[ n\left(
1-\alpha \right) +\left( \alpha -\delta \right) \right] \qquad \left(
n=2,3,\ldots \right) .
\end{equation*}
\end{cor}

Letting $\delta =0,\;\beta \rightarrow \infty $ in Theorem $\ref{thm1}$, we
have the following coefficient bounds for close-to-convex functions of order
$\alpha $.

\begin{cor}
Let $\alpha $\ be a real number such that $0\leq \alpha <1$\ and let the
function $f\in \mathcal{A}$ be defined by $(\ref{1.1})$. If $f\in \mathcal{C}%
(\alpha )$, then%
\begin{equation*}
\left\vert a_{n}\right\vert \leq n\left( 1-\alpha \right) +\alpha \qquad
\left( n=2,3,\ldots \right) .
\end{equation*}
\end{cor}

Letting $\alpha =\delta =0,\;\beta \rightarrow \infty $ in Theorem $\ref%
{thm1}$, we have the well-known coefficient bounds for close-to-convex
functions.

\begin{cor}
\cite{Re} Let the function $f\in \mathcal{A}$ be defined by $(\ref{1.1})$.
If $f\in \mathcal{C}$, then%
\begin{equation*}
\left\vert a_{n}\right\vert \leq n\qquad \left( n=2,3,\ldots \right) .
\end{equation*}
\end{cor}

\begin{thm}
\label{thm2}Let $\alpha ,\beta $\ and $\delta $\ be real numbers such that $%
0\leq \alpha ,\delta <1<\beta $\ and let the function $f\in \mathcal{A}$ be
defined by $(\ref{1.1})$. If $f\in \mathcal{B}_{g}\left( \alpha ,\beta ;\rho
\right) $, then%
\begin{eqnarray}
\left\vert a_{n}\right\vert  &\leq &\left\{ \frac{\prod\limits_{k=2}^{n}%
\left[ k-2+\frac{2\left( \beta -\delta \right) }{\pi }\sin \frac{\pi \left(
1-\delta \right) }{\beta -\delta }\right] }{n!}\right.   \notag \\
&&\left. +\frac{2\left( \beta -\alpha \right) }{n\pi }\sin \frac{\pi \left(
1-\alpha \right) }{\beta -\alpha }\left( 1+\sum_{j=1}^{n-2}\frac{%
\prod\limits_{k=2}^{n-j}\left[ k-2+\frac{2\left( \beta -\delta \right) }{\pi
}\sin \frac{\pi \left( 1-\delta \right) }{\beta -\delta }\right] }{\left(
n-j-1\right) !}\right) \right\}   \notag \\
&&\times \frac{\left( 1+\rho \right) \left( 2+\rho \right) }{\left( n+\rho
\right) \left( n+1+\rho \right) }\qquad \left( n=2,3,\ldots \right) ,
\label{3.1}
\end{eqnarray}%
where $\rho \in \mathbb{R}\backslash \left( -\infty ,-1\right] .$
\end{thm}

\begin{proof}
Let the function $f\in \mathcal{A}$ be given by $(\ref{1.1})$. Also let%
\begin{equation*}
\varphi (z)=z+\sum_{n=2}^{\infty }\varphi _{n}z^{n}\in \mathcal{S}_{g}\left(
\alpha ,\beta \right) .
\end{equation*}%
We then deduce from Definition $\ref{dfn4}$ that%
\begin{equation*}
a_{n}=\frac{\left( 1+\rho \right) \left( 2+\rho \right) }{\left( n+\rho
\right) \left( n+1+\rho \right) }\varphi _{n}\qquad \left( n=2,3,\ldots
;\;\rho \in \mathbb{R}\backslash \left( -\infty ,-1\right] \right) .
\end{equation*}%
Thus, by using Theorem $\ref{thm1}$ in conjunction with the above equality,
we have assertion $\left( \ref{3.1}\right) $ of Theorem $\ref{thm2}$.
\end{proof}

Letting $\beta \rightarrow \infty $ in Theorem $\ref{thm2}$, we have the
following consequence.

\begin{cor}
Let $\alpha $\ and $\delta $\ be real numbers such that $0\leq \alpha
,\delta <1$\ and let the function $f\in \mathcal{A}$ be defined by $(\ref%
{1.1})$. If $f\in \mathcal{B}_{g}\left( \alpha ;\rho \right) $, then%
\begin{equation*}
\left\vert a_{n}\right\vert \leq \left\{ \frac{\prod\limits_{k=2}^{n}\left(
k-2\delta \right) }{n!}+\frac{2\left( 1-\alpha \right) }{n}\left(
1+\sum_{j=1}^{n-2}\frac{\prod\limits_{k=2}^{n-j}\left( k-2\delta \right) }{%
\left( n-j-1\right) !}\right) \right\} \frac{\left( 1+\rho \right) \left(
2+\rho \right) }{\left( n+\rho \right) \left( n+1+\rho \right) }\qquad
\left( n=2,3,\ldots \right) ,
\end{equation*}%
where $\rho \in \mathbb{R}\backslash \left( -\infty ,-1\right] .$
\end{cor}

Letting $\delta =0,\;\beta \rightarrow \infty $ in Theorem $\ref{thm2}$, we
have the following consequence.

\begin{cor}
Let $\alpha $\ be a real number such that $0\leq \alpha <1$\ and let the
function $f\in \mathcal{A}$ be defined by $(\ref{1.1})$. If $f\in \mathcal{H}%
_{g}\left( \alpha ;\rho \right) $, then%
\begin{equation*}
\left\vert a_{n}\right\vert \leq \left[ n\left( 1-\alpha \right) +\alpha %
\right] \frac{\left( 1+\rho \right) \left( 2+\rho \right) }{\left( n+\rho
\right) \left( n+1+\rho \right) }\qquad \left( n=2,3,\ldots \right) ,
\end{equation*}%
where $\rho \in \mathbb{R}\backslash \left( -\infty ,-1\right] .$
\end{cor}

Letting $\alpha =\delta =0,\;\beta \rightarrow \infty $ in Theorem $\ref%
{thm2}$, we have the following consequence.

\begin{cor}
Let the function $f\in \mathcal{A}$ be defined by $(\ref{1.1})$. If $f\in
\mathcal{M}_{g}\left( \rho \right) $, then%
\begin{equation*}
\left\vert a_{n}\right\vert \leq n\frac{\left( 1+\rho \right) \left( 2+\rho
\right) }{\left( n+\rho \right) \left( n+1+\rho \right) }\qquad \left(
n=2,3,\ldots \right) ,
\end{equation*}%
where $\rho \in \mathbb{R}\backslash \left( -\infty ,-1\right] .$
\end{cor}

\end{document}